\documentclass[12pt, reqno]{amsart}
\UseRawInputEncoding
\usepackage{amsmath, amsthm, amscd, amsfonts, amssymb, graphicx, color}
\usepackage[bookmarksnumbered, colorlinks, plainpages]{hyperref}
\input{mathrsfs.sty}

\hypersetup{colorlinks=true,linkcolor=red, anchorcolor=green,
citecolor=cyan, urlcolor=red, filecolor=magenta, pdftoolbar=true}

\newtheorem{theorem}{Theorem}[section]

\newtheorem{corollary}[theorem]{Corollary}
\theoremstyle{definition}
\newtheorem{definition}[theorem]{Definition}

\theoremstyle{remark}
\newtheorem{remark}[theorem]{Remark}
\numberwithin{equation}{section}

\begin{document}
\title[On the weighted contraharmonic means]
{On the weighted contraharmonic means}

\author[A.~Zamani]
{Ali Zamani}

\address{Department of Mathematics, Farhangian University, Tehran, Iran
\&
School of Mathematics and Computer Sciences, Damghan University, P.O.BOX 36715-364, Damghan, Iran}
\email{zamani.ali85@yahoo.com}

\subjclass[2020]{46L05; 47A63; 47A64; 47B65.}
\keywords{$C^*$-algebra; positive definite; contraharmonic mean; operator inequality.}
\begin{abstract}
Let $\mathscr{A}$ be a unital $C^*$-algebra with unit $e$ and let $\nu\in(0, 1)$.
We introduce the concept of the $\nu$-weighted contraharmonic of two positive definite elements $a$ and $b$ of $\mathscr{A}$ by
\begin{align*}
{C}_{\nu}(a, b):= (1-\nu)\nu^{-1}b + \nu (1-\nu)^{-1}a - \left((1-\nu)a^{-1}+\nu b^{-1}\right)^{-1}.
\end{align*}
We show that
\begin{align*}
{C}_{\nu}(a, b)= \displaystyle{\max_{x+y=e}}\left\{(1-\nu)^{-1}\left(\nu a - x^*ax\right) + \nu^{-1}\left((1-\nu)b - y^*by\right)\right\},
\end{align*}
and then apply it to present some properties of this weighted mean.
\end{abstract}
\maketitle
\section{Introduction and Preliminaries}
The theory of (weighted) means for numbers is a classical and very well developed area in mathematical analysis
(see, e.g., \cite[Chapters II-III]{Hardy.Littlewood.Polya.1988}).
A mean of positive scalars $\alpha$ and $\beta$ may be introduced in many different ways.
One of the most important is a concept of the Gini--Beckenbach--Lehmer mean
(\cite{Beckenbach.1950, Gini.1938, Lehmer.1971}):
\begin{align*}
M_{s}(\alpha, \beta)=\frac{\alpha^s+\beta^s}{\alpha^{s-1}+\beta^{s-1}}.
\end{align*}
Notice that the harmonic mean ($H$), geometric mean ($G$),
arithmetic mean ($A$) and contraharmonic mean ($C$), which is frequently used
in this paper, can be associated with this mean, respectively,
by letting $s=0$, $s=\frac{1}{2}$, $s=1$ and $s=2$. That is,
\begin{align*}
&H(\alpha, \beta)=M_{0}(\alpha, \beta)=\frac{2\alpha\beta}{\alpha+\beta},
\\& G(\alpha, \beta)=M_{\frac{1}{2}}(\alpha, \beta)=\sqrt{\alpha\beta},
\\& A(\alpha, \beta)=M_{1}(\alpha, \beta)= \frac{\alpha+\beta}{2},
\\& C(\alpha, \beta)=M_{2}(\alpha, \beta)= \frac{\alpha^2+\beta^2}{\alpha+\beta}.
\end{align*}
Let $\mathscr{A}$ be a $C^*$-algebra.
An element $a$ of $\mathscr{A}$ is positive, in short $0\leq a$, if $a = b^*b$ for some $b\in\mathscr{A}$.
If $0\leq a$, then we denote by $a^{1/2}$ the unique positive square root of $a$.
If $a$ and $b$ are self-adjoint elements
of $\mathscr{A}$ such that $0\leq a-b$, we write $b\leq a$.
An element $a$ of $\mathscr{A}$ is also said to be positive definite if $a$ is positive and invertible.
Averaging operations are of interest in the context of von Neumann algebras and $C^*$-algebras as well,
and various notions of (weighted) means of positive definite elements have been studied
(see \cite{Kubo.Ando.1980, Molnar.2017} and the references therein).

Let $\nu\in(0, 1)$. For two positive definite elements $a$ and $b$ of $\mathscr{A}$
the ($\nu$-weighted) harmonic mean $H_{\nu}$, ($\nu$-weighted) geometric mean $G_{\nu}$ and ($\nu$-weighted) arithmetic mean $A_{\nu}$
are defined by
\begin{align*}
&H_{\nu}(a, b)=\left((1-\nu)a^{-1}+\nu b^{-1}\right)^{-1},
\\& G_{\nu}(a, b)=a^{\frac{1}{2}}\left(a^{-\frac{1}{2}}ba^{-\frac{1}{2}}\right)^{\nu}a^{\frac{1}{2}},
\\& A_{\nu}(a, b)=(1-\nu)a+\nu b.
\end{align*}
In this paper, inspired by the definition for the contraharmonic
mean of matrices \cite{Anderson.M.M.T.1987, Green.Morley.1987},
we introduce the concept of the $\nu$-weighted contraharmonic
mean in the setting of $C^*$-algebras.
We investigate some properties of this weighted mean and
prove inequalities involving it.
\section{Results}
As we have already mentioned, the contraharmonic mean of
two positive scalars is defined by the formula
\begin{align*}
C(\alpha, \beta)=\frac{\alpha^2+\beta^2}{\alpha+\beta}.
\end{align*}
This may be rewritten as
\begin{align*}
C(\alpha, \beta)
= A(2\beta, 2\alpha)-H(\alpha, \beta).
\end{align*}
This motivates the following definition.
\begin{definition}\label{d.02}
Let $\mathscr{A}$ be a unital $C^*$-algebra with unit $e$ and let $\nu\in(0, 1)$.
The $\nu$-weighted contraharmonic mean of two positive definite elements
$a$ and $b$ of $\mathscr{A}$ is defined by
\begin{align*}
{C}_{\nu}(a, b)= {A}_{\nu}\left(\nu^{-1}b, (1-\nu)^{-1}a\right) - {H}_{\nu}(a, b).
\end{align*}
\end{definition}
\begin{remark}\label{R.0004011}
In the sequel, $a, b, c$ and $d$ denote positive definite elements of a unital $C^*$-algebra $\mathscr{A}$ with unit $e$.
\end{remark}
\begin{remark}\label{R.000401}
It is easy to see that ${H}_{\nu}(a, b) \leq (1-\nu)^{-1}a$ and ${H}_{\nu}(a, b) \leq \nu^{-1}b$.
Thus, we obtain $0\leq {C}_{\nu}(a, b)$.
Also, since $0\leq {H}_{\nu}(a, b)$, we have
\begin{align}\label{R.000401.I.1}
{C}_{\nu}(a, b)\leq {A}_{\nu}\left(\nu^{-1}b, (1-\nu)^{-1}a\right).
\end{align}
\end{remark}
\begin{remark}\label{R.0004801}
The following properties of the weighted contraharmonic mean are obvious:
\begin{itemize}
\item[(i)] ${C}_{\nu}(a, b)= {C}_{1-\nu}(b, a)$.
\item[(ii)] ${C}_{\nu}(a, a) = \frac{3\nu^2-3\nu+1}{\nu-\nu^2}a$.
\item[(iii)] ${C}_{\nu}(\alpha e, \beta e) = {C}_{\nu}(\alpha, \beta)e$ for any $\alpha, \beta>0$.
\item[(iv)] ${C}_{\nu}(ra, rb) = r{C}_{\nu}(a, b)$ for any $r>0$.
\end{itemize}
\end{remark}
If $\alpha$ and $\beta$ are two positive scalars, then the contraharmonic mean $C(\alpha, \beta)$
can be stated by the solution of the following variational problem:
\begin{align*}
C(\alpha, \beta)= \displaystyle{\max_{s+t=1}}\left\{\alpha-2\alpha s^2 + \beta -2\beta t^2\right\}.
\end{align*}
Motivated by this expression for the contraharmonic mean of scalars, we establish the following theorem.
\begin{theorem}\label{T.01}
The following expression holds:
\begin{align*}
{C}_{\nu}(a, b)= \displaystyle{\max_{x+y=e}}\left\{(1-\nu)^{-1}\left(\nu a - x^*ax\right) + \nu^{-1}\left((1-\nu)b - y^*by\right)\right\}.
\end{align*}
\end{theorem}
\begin{proof}
Note first that, by direct computations we have
\begingroup\makeatletter\def\f@size{10}\check@mathfonts
\begin{align}\label{T.01.I.1.1}
a\left(a+\nu^{-1}(1-\nu)b\right)^{-1}b=\nu{H}_{\nu}(a, b)=b\left(a+\nu^{-1}(1-\nu)b\right)^{-1}a
\end{align}
\endgroup
and
\begingroup\makeatletter\def\f@size{9}\check@mathfonts
\begin{align}\label{T.01.I.1.2}
&\nu^{-1}(1-\nu)a^{-\frac{1}{2}}ba^{-\frac{1}{2}}
-\left(e+\nu(1-\nu)^{-1}a^{\frac{1}{2}}b^{-1}a^{\frac{1}{2}}\right)^{-1}\nonumber
\\& \qquad = \left(e+\nu^{-1}(1-\nu)a^{-\frac{1}{2}}ba^{-\frac{1}{2}}\right)^{-\frac{1}{2}}
\left(\nu^{-1}(1-\nu)a^{-\frac{1}{2}}ba^{-\frac{1}{2}}\right)^{2}
\left(e+\nu^{-1}(1-\nu)a^{-\frac{1}{2}}ba^{-\frac{1}{2}}\right)^{-\frac{1}{2}}.
\end{align}
\endgroup
Set $z=\nu^{-1}(1-\nu)\left(a+\nu^{-1}(1-\nu)b\right)^{-1}b$ and $w=\left(a+\nu^{-1}(1-\nu)b\right)^{-1}a$.
Then $z+w=e$. By \eqref{T.01.I.1.1}, we have
\begingroup\makeatletter\def\f@size{10}\check@mathfonts
\begin{align*}
&\displaystyle{\max_{x+y=e}}\left\{(1-\nu)^{-1}\left(\nu a - x^*ax\right) + \nu^{-1}\left((1-\nu)b - y^*by\right)\right\}
\\& \qquad \geq (1-\nu)^{-1}\left(\nu a - z^*az\right) + \nu^{-1}\left((1-\nu)b - w^*bw\right)
\\& \qquad = (1-\nu)^{-1}\nu a - \nu^{-2}(1-\nu)b\left(a+\nu^{-1}(1-\nu)b\right)^{-1}a\left(a+\nu^{-1}(1-\nu)b\right)^{-1}b
\\& \qquad \qquad +\nu^{-1}(1-\nu)b - \nu^{-1}a\left(a+\nu^{-1}(1-\nu)b\right)^{-1}b\left(a+\nu^{-1}(1-\nu)b\right)^{-1}a
\\& \qquad = {A}_{\nu}\left(\nu^{-1}b, (1-\nu)^{-1}a\right) - \nu^{-1}(1-\nu)b\left(a+\nu^{-1}(1-\nu)b\right)^{-1}{H}_{\nu}(a, b)
\\& \qquad \qquad -a\left(a+\nu^{-1}(1-\nu)b\right)^{-1}{H}_{\nu}(a, b)
\\& \qquad = {A}_{\nu}\left(\nu^{-1}b, (1-\nu)^{-1}a\right) - \left(\nu^{-1}(1-\nu)b+a\right)\left(a+\nu^{-1}(1-\nu)b\right)^{-1}{H}_{\nu}(a, b)
\\& \qquad ={A}_{\nu}\left(\nu^{-1}b, (1-\nu)^{-1}a\right) - {H}_{\nu}(a, b)={C}_{\nu}(a, b),
\end{align*}
\endgroup
and hence
\begingroup\makeatletter\def\f@size{10}\check@mathfonts
\begin{align}\label{T.01.I.2}
{C}_{\nu}(a, b)\leq \displaystyle{\max_{x+y=e}}\left\{(1-\nu)^{-1}\left(\nu a - x^*ax\right) + \nu^{-1}\left((1-\nu)b - y^*by\right)\right\}.
\end{align}
\endgroup
Now, suppose $x, y\in\mathscr{A}$ with $x+y=e$. Let us put
\begingroup\makeatletter\def\f@size{9}\check@mathfonts
\begin{align*}
h := \left(e+\nu^{-1}(1-\nu)a^{-\frac{1}{2}}ba^{-\frac{1}{2}}\right)^{\frac{1}{2}}a^{\frac{1}{2}}xa^{-\frac{1}{2}}
-\nu^{-1}(1-\nu)a^{-\frac{1}{2}}ba^{-\frac{1}{2}}
\left(e+\nu^{-1}(1-\nu)a^{-\frac{1}{2}}ba^{-\frac{1}{2}}\right)^{-\frac{1}{2}}.
\end{align*}
\endgroup
By exploiting \eqref{T.01.I.1.2} we have
\begingroup\makeatletter\def\f@size{9}\check@mathfonts
\begin{align*}
&a^{\frac{1}{2}}\left(e+\nu(1-\nu)^{-1}a^{\frac{1}{2}}b^{-1}a^{\frac{1}{2}}\right)^{-1}a^{\frac{1}{2}}+a^{\frac{1}{2}}h^*ha^{\frac{1}{2}}
\\& \qquad = a^{\frac{1}{2}}\left(e+\nu(1-\nu)^{-1}a^{\frac{1}{2}}b^{-1}a^{\frac{1}{2}}\right)^{-1}a^{\frac{1}{2}}
+x^*a^{\frac{1}{2}}\left(e+\nu^{-1}(1-\nu)a^{-\frac{1}{2}}ba^{-\frac{1}{2}}\right)a^{\frac{1}{2}}x
\\& \qquad \quad -x^*a^{\frac{1}{2}}\left(e+\nu^{-1}(1-\nu)a^{-\frac{1}{2}}ba^{-\frac{1}{2}}\right)^{\frac{1}{2}}
\left(\nu^{-1}(1-\nu)a^{-\frac{1}{2}}ba^{-\frac{1}{2}}\right)\left(e+\nu^{-1}(1-\nu)a^{-\frac{1}{2}}ba^{-\frac{1}{2}}\right)^{-\frac{1}{2}}a^{\frac{1}{2}}
\\& \qquad \quad -a^{\frac{1}{2}}\left(e+\nu^{-1}(1-\nu)a^{-\frac{1}{2}}ba^{-\frac{1}{2}}\right)^{-\frac{1}{2}}
\left(\nu^{-1}(1-\nu)a^{-\frac{1}{2}}ba^{-\frac{1}{2}}\right)\left(e+\nu^{-1}(1-\nu)a^{-\frac{1}{2}}ba^{-\frac{1}{2}}\right)^{\frac{1}{2}}a^{\frac{1}{2}}x
\\& \qquad \quad +a^{\frac{1}{2}}\left(e+\nu^{-1}(1-\nu)a^{-\frac{1}{2}}ba^{-\frac{1}{2}}\right)^{-\frac{1}{2}}
\left(\nu^{-1}(1-\nu)a^{-\frac{1}{2}}ba^{-\frac{1}{2}}\right)^{2}
\left(e+\nu^{-1}(1-\nu)a^{-\frac{1}{2}}ba^{-\frac{1}{2}}\right)^{-\frac{1}{2}}a^{\frac{1}{2}}
\\& \qquad = a^{\frac{1}{2}}\left(e+\nu(1-\nu)^{-1}a^{\frac{1}{2}}b^{-1}a^{\frac{1}{2}}\right)^{-1}a^{\frac{1}{2}}
+x^*ax
\\& \qquad \quad +a^{\frac{1}{2}}\left(e-a^{-\frac{1}{2}}x^*a^{\frac{1}{2}}\right)\left(\nu^{-1}(1-\nu)a^{-\frac{1}{2}}ba^{-\frac{1}{2}}\right)
\left(e-a^{\frac{1}{2}}xa^{-\frac{1}{2}}\right)a^{\frac{1}{2}}
-\nu^{-1}(1-\nu)b
\\& \qquad \quad +a^{\frac{1}{2}}\left(e+\nu^{-1}(1-\nu)a^{-\frac{1}{2}}ba^{-\frac{1}{2}}\right)^{-\frac{1}{2}}
\left(\nu^{-1}(1-\nu)a^{-\frac{1}{2}}ba^{-\frac{1}{2}}\right)^{2}
\left(e+\nu^{-1}(1-\nu)a^{-\frac{1}{2}}ba^{-\frac{1}{2}}\right)^{-\frac{1}{2}}a^{\frac{1}{2}}
\\& \qquad = a^{\frac{1}{2}}\left(e+\nu(1-\nu)^{-1}a^{\frac{1}{2}}b^{-1}a^{\frac{1}{2}}\right)^{-1}a^{\frac{1}{2}}
+x^*ax
\\& \qquad \quad +a^{\frac{1}{2}}\left(a^{-\frac{1}{2}}y^*a^{\frac{1}{2}}\right)\left(\nu^{-1}(1-\nu)a^{-\frac{1}{2}}ba^{-\frac{1}{2}}\right)
\left(a^{\frac{1}{2}}ya^{-\frac{1}{2}}\right)a^{\frac{1}{2}}
-\nu^{-1}(1-\nu)b
\\& \qquad \quad +\nu^{-1}(1-\nu)b
-a^{\frac{1}{2}}\left(e+\nu(1-\nu)^{-1}a^{\frac{1}{2}}b^{-1}a^{\frac{1}{2}}\right)^{-1}a^{\frac{1}{2}}
\\& \qquad = x^*ax+\nu^{-1}(1-\nu)y^*by,
\end{align*}
\endgroup
and wherefrom
\begingroup\makeatletter\def\f@size{9}\check@mathfonts
\begin{align}\label{T.01.I.3}
a^{\frac{1}{2}}\left(e+\nu(1-\nu)^{-1}a^{\frac{1}{2}}b^{-1}a^{\frac{1}{2}}\right)^{-1}a^{\frac{1}{2}}
= x^*ax+\nu^{-1}(1-\nu)y^*by-a^{\frac{1}{2}}h^*ha^{\frac{1}{2}}.
\end{align}
\endgroup
Since $0\leq (1-\nu)^{-1}a^{\frac{1}{2}}h^*ha^{\frac{1}{2}}$, by \eqref{T.01.I.3} we obtain
\begingroup\makeatletter\def\f@size{10}\check@mathfonts
\begin{align*}
&(1-\nu)^{-1}\left(\nu a - x^*ax\right) + \nu^{-1}\left((1-\nu)b - y^*by\right)
\\& \qquad \leq (1-\nu)^{-1}\left(\nu a - x^*ax\right) + \nu^{-1}\left((1-\nu)b - y^*by\right) + (1-\nu)^{-1}a^{\frac{1}{2}}h^*ha^{\frac{1}{2}}
\\& \qquad = (1-\nu)\nu^{-1}b + \nu (1-\nu)^{-1}a
-(1-\nu)^{-1}a^{\frac{1}{2}}\left(e + \nu(1-\nu)^{-1} a^{\frac{1}{2}}b^{-1}a^{\frac{1}{2}}\right)^{-1}a^{\frac{1}{2}}
\\& \qquad= (1-\nu)\nu^{-1}b + \nu (1-\nu)^{-1}a -(1-\nu)^{-1}\left(a^{-1} + \nu(1-\nu)^{-1} b^{-1}\right)^{-1}
\\& \qquad= (1-\nu)\nu^{-1}b + \nu (1-\nu)^{-1}a - \left((1-\nu)a^{-1} + \nu b^{-1}\right)^{-1}
\\& \qquad={A}_{\nu}\left(\nu^{-1}b, (1-\nu)^{-1}a\right) - {H}_{\nu}(a, b),
\end{align*}
\endgroup
and so
\begingroup\makeatletter\def\f@size{10}\check@mathfonts
\begin{align}\label{T.01.I.4}
(1-\nu)^{-1}\left(\nu a - x^*ax\right) + \nu^{-1}\left((1-\nu)b - y^*by\right)\leq {C}_{\nu}(a, b) \qquad (x+y=e).
\end{align}
\endgroup
It follows from \eqref{T.01.I.4} that
\begingroup\makeatletter\def\f@size{10}\check@mathfonts
\begin{align}\label{T.01.I.5}
\displaystyle{\max_{x+y=e}}\left\{(1-\nu)^{-1}\left(\nu a - x^*ax\right) + \nu^{-1}\left((1-\nu)b - y^*by\right)\right\}\leq {C}_{\nu}(a, b).
\end{align}
\endgroup
Now, by \eqref{T.01.I.2} and \eqref{T.01.I.5}, we deduce the desired result.
\end{proof}
As a consequence of Theorem \ref{T.01}, we have the following result.
\begin{corollary}\label{C.02072}
Let $\mu\in(0, 1)$. Then
\begin{align*}
{C}_{\nu}\big({A}_{\mu}(a, b), {A}_{\mu}(c, d)\big) \leq {A}_{\mu}\big({C}_{\nu}(a, c), {C}_{\nu}(b, d)\big).
\end{align*}
\end{corollary}
\begin{proof}
Let $x+y=e$.
By Theorem \ref{T.01} we have
\begingroup\makeatletter\def\f@size{11}\check@mathfonts
\begin{align*}
{A}_{\mu}\big({C}_{\nu}(a, c), {C}_{\nu}(b, d)\big)&=(1-\mu){C}_{\nu}(a, c)+\mu{C}_{\nu}(b, d)
\\&\geq (1-\mu)(1-\nu)^{-1}\left(\nu a - x^*ax\right) + (1-\mu)\nu^{-1}\left((1-\nu)c - y^*cy\right)
\\& \qquad + \mu(1-\nu)^{-1}\left(\nu b - x^*bx\right) + \mu\nu^{-1}\left((1-\nu)d - y^*dy\right)
\\& = (1-\nu)^{-1}\Big(\nu ((1-\mu)a+\mu b) - x^*((1-\mu)a+\mu b)x\Big)
\\& \qquad + \nu^{-1}\Big((1-\nu)((1-\mu)c+\mu d) - y^*((1-\mu)c+\mu d)y\Big),
\end{align*}
\endgroup
and hence
\begingroup\makeatletter\def\f@size{9}\check@mathfonts
\begin{align*}
(1-\nu)^{-1}\Big(\nu {A}_{\mu}(a, b) - x^*{A}_{\mu}(a, b)x\Big)
+ \nu^{-1}\Big((1-\nu){A}_{\mu}(c, d) - y^*{A}_{\mu}(c, d)y\Big)\leq {A}_{\mu}\left({C}_{\nu}(a, c), {C}_{\nu}(b, d)\right).
\end{align*}
\endgroup
Thus,
\begingroup\makeatletter\def\f@size{8}\check@mathfonts
\begin{align*}
\displaystyle{\max_{x+y=e}}\left\{(1-\nu)^{-1}\Big(\nu {A}_{\mu}(a, b) - x^*{A}_{\mu}(a, b)x\Big)
+ \nu^{-1}\Big((1-\nu){A}_{\mu}(c, d) - y^*{A}_{\mu}(c, d)y\Big)\right\}
\leq {A}_{\mu}\big({C}_{\nu}(a, c), {C}_{\nu}(b, d)\big).
\end{align*}
\endgroup
Now, from Theorem \ref{T.01}
we obtain ${C}_{\nu}\big({A}_{\mu}(a, b), {A}_{\mu}(c, d)\big) \leq {A}_{\mu}\big({C}_{\nu}(a, c), {C}_{\nu}(b, d)\big)$.
\end{proof}
Another consequence of Theorem \ref{T.01} can be stated as follows.
\begin{corollary}\label{C.02073}
Let $z$ be an invertible element of $\mathscr{A}$. Then
\begin{align*}
{C}_{\nu}\big(z^*az, z^*bz\big) = z^*{C}_{\nu}(a, b)z.
\end{align*}
\end{corollary}
\begin{proof}
Let $x+y=e$. Put ${x}_{_0}= zxz^{-1}$ and ${y}_{_0}= zyz^{-1}$. Then ${x}_{_0}+{y}_{_0}=e$. So, by Theorem \ref{T.01}, we have
\begingroup\makeatletter\def\f@size{11}\check@mathfonts
\begin{align*}
z^*{C}_{\nu}(a, b)z&
= z^*\left(\displaystyle{\max_{x+y=e}}\left\{(1-\nu)^{-1}\left(\nu a - x^*ax\right) + \nu^{-1}\left((1-\nu)b - y^*by\right)\right\}\right)z
\\& \geq z^*\left((1-\nu)^{-1}\left(\nu a - {x}_{_0}^*a{x}_{_0}\right) + \nu^{-1}\left((1-\nu)b - {y}_{_0}^*b{y}_{_0}\right)\right)z
\\& = (1-\nu)^{-1}\left(\nu (z^*az) - x^*(z^*az)x\right) + \nu^{-1}\left((1-\nu)(z^*bz) - y^*(z^*bz)y\right),
\end{align*}
\endgroup
and so
\begingroup\makeatletter\def\f@size{11}\check@mathfonts
\begin{align*}
(1-\nu)^{-1}\left(\nu (z^*az) - x^*(z^*az)x\right) + \nu^{-1}\left((1-\nu)(z^*bz) - y^*(z^*bz)y\right)\leq z^*{C}_{\nu}(a, b)z.
\end{align*}
\endgroup
Therefore,
\begingroup\makeatletter\def\f@size{10}\check@mathfonts
\begin{align*}
\displaystyle{\max_{x+y=e}}\left\{(1-\nu)^{-1}\left(\nu (z^*az) - x^*(z^*az)x\right) + \nu^{-1}\left((1-\nu)(z^*bz) - y^*(z^*bz)y\right)\right\}
\leq z^*{C}_{\nu}(a, b)z.
\end{align*}
\endgroup
Now, by Theorem \ref{T.01}, we conclude that ${C}_{\nu}\big(z^*az, z^*bz\big) \leq z^*{C}_{\nu}(a, b)z$.
By a similar argument, we get $z^*{C}_{\nu}(a, b)z\leq{C}_{\nu}\big(z^*az, z^*bz\big)$ and
the proof is completed.
\end{proof}
Our next result reads as follows.
\begin{theorem}\label{T.021073}
Let $\mu\in (0, 1)$. Then
\begin{align*}
{C}_{\nu}\big(a, {A}_{\mu}(a, b)\big) \leq {A}_{\mu}\left(\frac{2\nu^2-2\nu+1}{\nu-\nu^2} a, {C}_{\nu}(a, b)\right).
\end{align*}
\end{theorem}
\begin{proof}
Let $x+y=e$. Using the inequality
\begin{align*}
0\leq\left(\sqrt{1-\nu}a^{\frac{1}{2}}-\frac{1}{\sqrt{1-\nu}}a^{\frac{1}{2}}x\right)^*
\left(\sqrt{1-\nu}a^{\frac{1}{2}}-\frac{1}{\sqrt{1-\nu}}a^{\frac{1}{2}}x\right),
\end{align*}
we obtain
\begin{align*}
0\leq (1-\nu)a-ax-x^*a+ (1-\nu)^{-1}x^*ax.
\end{align*}
Since $0\leq \nu a$, from the above inequality we get
\begin{align}\label{T.021073.I.1}
0\leq a-ax-x^*a+ (1-\nu)^{-1}x^*ax.
\end{align}
Let us put $\gamma :=\frac{2\nu^2-2\nu+1}{\nu-\nu^2}$.
By Theorem \ref{T.01} and \eqref{T.021073.I.1} we have
\begingroup\makeatletter\def\f@size{10}\check@mathfonts
\begin{align*}
&(1-\nu)^{-1}\left(\nu a - x^*ax\right) + \nu^{-1}\Big((1-\nu){A}_{\mu}(a, b) - y^*{A}_{\mu}(a, b)y\Big)
\\& \quad = (1-\nu)^{-1}\nu a - (1-\nu)^{-1}x^*ax + (1-\mu)\nu^{-1}(1-\nu)a + \mu\nu^{-1}(1-\nu)b
\\& \quad \qquad -(1-\mu)\nu^{-1}y^*ay - \mu\nu^{-1}y^*by
\\& \quad = (1-\mu)\frac{2\nu^2-2\nu+1}{\nu-\nu^2}a+(1-\nu)^{-1}\nu a+(1-\mu)\left(\nu^{-1}(1-\nu)-\frac{2\nu^2-2\nu+1}{\nu-\nu^2}\right)a
\\& \quad \qquad - (1-\nu)^{-1}x^*ax + \mu\nu^{-1}(1-\nu)b -(1-\mu)\nu^{-1}(e-x^*)a(e-x) - \mu\nu^{-1}y^*by
\\& \quad = (1-\mu)\gamma a+\mu(1-\nu)^{-1}\nu a - (1-\nu)^{-1}x^*ax + \mu\nu^{-1}(1-\nu)b
\\& \quad \qquad -(1-\mu)\nu^{-1}a + (1-\mu)\nu^{-1}(ax+x^*a)-(1-\mu)\nu^{-1}x^*ax -\mu\nu^{-1}y^*by
\\& \quad = (1-\mu)\gamma a + \mu\Big((1-\nu)^{-1}\left(\nu a - x^*ax\right) + \nu^{-1}\left((1-\nu)b - y^*by\right)\Big)
\\& \qquad \qquad -(1-\mu)\nu^{-1}\Big(a-ax-x^*a+ (1-\nu)^{-1}x^*ax\Big)
\\& \quad \leq (1-\mu)\gamma a + \mu{C}_{\nu}(a, b) = {A}_{\mu}\big(\gamma a, {C}_{\nu}(a, b)\big),
\end{align*}
\endgroup
and so
\begingroup\makeatletter\def\f@size{10}\check@mathfonts
\begin{align*}
(1-\nu)^{-1}\left(\nu a - x^*ax\right) + \nu^{-1}\Big((1-\nu){A}_{\mu}(a, b) - y^*{A}_{\mu}(a, b)y\Big)
\leq {A}_{\mu}\big(\gamma a, {C}_{\nu}(a, b)\big).
\end{align*}
\endgroup
From this it follows that
\begingroup\makeatletter\def\f@size{10}\check@mathfonts
\begin{align*}
\displaystyle{\max_{x+y=e}}\left\{(1-\nu)^{-1}\left(\nu a - x^*ax\right) + \nu^{-1}\Big((1-\nu){A}_{\mu}(a, b) - y^*{A}_{\mu}(a, b)y\Big)\right\}
\leq {A}_{\mu}\big(\gamma a, {C}_{\nu}(a, b)\big),
\end{align*}
\endgroup
and by Theorem \ref{T.01} we conclude that
${C}_{\nu}\big(a, {A}_{\mu}(a, b)\big) \leq {A}_{\mu}\big(\gamma a, {C}_{\nu}(a, b)\big)$.
\end{proof}
Here, we state an inequality for non-zero positive linear functional.
\begin{corollary}\label{C.004}
Let $\varphi$ be a non-zero positive linear functional on $\mathscr{A}$. Then
\begin{align*}
{C}_{\nu}\big(\varphi(a), \varphi(b)\big)\leq \varphi\big({C}_{\nu}(a, b)\big).
\end{align*}
\end{corollary}
\begin{proof}
Set ${x}_{_0}=\frac{(1-\nu)\varphi(b)}{\varphi\left(\nu a+(1-\nu)b\right)}e$ and ${y}_{_0}=\frac{\nu\varphi(a)}{\varphi\left(\nu a+(1-\nu)b\right)}e$.
Then ${x}_{_0}+{y}_{_0}=e$. Hence, by Theorem \ref{T.01}, we have
\begin{align*}
\varphi\big({C}_{\nu}(a, b)\big)&= \varphi\left(\displaystyle{\max_{x+y=e}}\left\{(1-\nu)^{-1}\left(\nu a - x^*ax\right) + \nu^{-1}\left((1-\nu)b - y^*by\right)\right\}\right)
\\& \geq \varphi\left((1-\nu)^{-1}\left(\nu a - {x}_{_0}^*a{x}_{_0}\right) + \nu^{-1}\left((1-\nu)b - {y}_{_0}^*b{y}_{_0}\right)\right)
\\& = (1-\nu)^{-1}\nu\varphi(a)+\nu^{-1}(1-\nu)\varphi(b)
\\& \qquad - \left(\frac{(1-\nu)\varphi^2(b)\varphi(a)}{\varphi^2\left(\nu a+(1-\nu)b\right)}
+ \frac{\nu\varphi^2(a)\varphi(b)}{\varphi^2\left(\nu a+(1-\nu)b\right)}\right)
\\& = (1-\nu)\nu^{-1}\varphi(b) + \nu(1-\nu)^{-1}\varphi(a) - \frac{\varphi(a)\varphi(b)}{\varphi\left(\nu a+(1-\nu)b\right)}
\\&= {A}_{\nu}\left(\nu^{-1}\varphi(b), (1-\nu)^{-1}\varphi(a)\right) - {H}_{\nu}\big(\varphi(a), \varphi(b)\big)
= {C}_{\nu}\big(\varphi(a), \varphi(b)\big).
\end{align*}
\end{proof}
Our next application of Theorem \ref{T.01} will establish a lower bounds for the $\nu$-weighted contraharmonic mean.
\begin{corollary}\label{T.007}
If $\alpha = \frac{\|b\|}{{A}_{\nu}(\|b\|, \|a\|)}$ and $\beta = \frac{\|b\|}{{A}_{\nu}(\|b\|, \|a\|)}$, then
\begin{align*}
{A}_{\nu}\left(\nu^{-1}b, (1-\nu)^{-1}a\right) - {A}_{\nu}\left(\alpha^2a, \beta^2b\right) \leq {C}_{\nu}(a, b).
\end{align*}
\end{corollary}
\begin{proof}
It is easy to see that $(1-\nu)\alpha\,e+\nu\beta\,e=e$ and hence by Theorem \ref{T.01}, we have
\begingroup\makeatletter\def\f@size{10}\check@mathfonts
\begin{align*}
{C}_{\nu}(a, b)&\geq (1-\nu)^{-1}\Big(\nu a - \big((1-\nu)\alpha\,e\big)^*a\big((1-\nu)\alpha\,e\big)\Big)
+ \nu^{-1}\Big((1-\nu)b - \big(\nu\beta\,e\big)^*b\big(\nu\beta\,e\big)\Big)
\\& = (1-\nu)^{-1}\nu a -(1-\nu)\alpha^2a + \nu^{-1}(1-\nu)b - \nu\beta^2b
\\& = {A}_{\nu}\left(\nu^{-1}b, (1-\nu)^{-1}a\right) - {A}_{\nu}\left(\alpha^2a, \beta^2b\right).
\end{align*}
\endgroup
\end{proof}
In the next theorem, we present a family of lower bounds for ${C}_{\nu}(a, b)$.
\begin{theorem}\label{C.02023}
Let $\lambda\in[0, 1]$. Then
\begin{align*}
(1-\nu)^{-1}\left(\nu - \lambda^2\right)a + \nu^{-1}\left(2\lambda-\lambda^2-\nu\right)b \leq {C}_{\nu}(a, b).
\end{align*}
In particular, $2((1-\nu)^{-\frac{1}{2}}-1)a\leq {C}_{\nu}(a, b)$ and $2(\nu^{-\frac{1}{2}}-1)b\leq {C}_{\nu}(a, b)$.
\end{theorem}
\begin{proof}
Since $\lambda e + (1-\lambda)e=e$, by Theorem \ref{T.01}, we have
\begin{align*}
{C}_{\nu}(a, b) &\geq(1-\nu)^{-1}\left(\nu a - \left(\lambda e\right)^*a\left(\lambda e\right)\right)
\\& \qquad + \nu^{-1}\left((1-\nu)b - \left((1-\lambda)e\right)^*b\left((1-\lambda)e\right)\right)
\\& = (1-\nu)^{-1}\left(\nu - \lambda^2\right)a + \nu^{-1}\left(1-\nu - (1-\lambda)^2\right)b
\\& = (1-\nu)^{-1}\left(\nu - \lambda^2\right)a + \nu^{-1}\left(2\lambda-\lambda^2-\nu\right)b.
\end{align*}
In particular, for $\lambda=1-\sqrt{1-\nu}$ and $\lambda=\sqrt{\nu}$, we have
\begin{align*}
{C}_{\nu}(a, b) \geq \nu^{-1}\left(2\sqrt{\nu}-2\nu\right)b = 2(\nu^{-\frac{1}{2}}-1)b
\end{align*}
and
\begin{align*}
{C}_{\nu}(a, b) \geq (1-\nu)^{-1}\left(2\nu - 2 +2\sqrt{1-\nu}\right)a=2((1-\nu)^{-\frac{1}{2}}-1)a.
\end{align*}
\end{proof}
The next assertion is interesting on its own right.
\begin{corollary}\label{T.0059}
There exists a contraction $z$ in $\mathscr{A}$ such that
\begin{align*}
{A}_{\nu}(b, a)=z^*{C}_{\nu}(a, b)z.
\end{align*}
\end{corollary}
\begin{proof}
By letting  $\lambda=\nu$ in Theorem \ref{C.02023}, we get
\begin{align}\label{C.02023.I.1}
{A}_{\nu}(b, a) \leq {C}_{\nu}(a, b).
\end{align}
Put $z=\left({C}_{\nu}(a, b)\right)^{-\frac{1}{2}}\left({A}_{\nu}(b, a)\right)^{\frac{1}{2}}$.
From \eqref{C.02023.I.1} it follows that
\begin{align*}
z^*z = \left({A}_{\nu}(b, a)\right)^{\frac{1}{2}}\left({C}_{\nu}(a, b)\right)^{-1}\left({A}_{\nu}(b, a)\right)^{\frac{1}{2}}\leq e,
\end{align*}
and hence $z$ is a contraction.
It is also easy to check that $z^*{C}_{\nu}(a, b)z= {A}_{\nu}(b, a)$.
\end{proof}
We close this paper with an inequality that refines the inequality \eqref{R.000401.I.1}.
\begin{theorem}\label{T.0071}
The following inequality holds:
\begin{align*}
{C}_{\nu}(a, b)\leq {A}_{\nu}\left(\nu^{-1}b, (1-\nu)^{-1}a\right)
- {H}_{\nu}\left({\left\|a^{-1}\right\|}^{-1}, {\left\|b^{-1}\right\|}^{-1}\right)e.
\end{align*}
\end{theorem}
\begin{proof}
Let $x+y=e$. Since ${\left\|c^{-1}\right\|}^{-1}e\leq c$ for any positive definite element $c\in\mathscr{A}$,
we have
\begingroup\makeatletter\def\f@size{8.2}\check@mathfonts
\begin{align*}
&(1-\nu)^{-1}\left(\nu a - x^*ax\right) + \nu^{-1}\left((1-\nu)b - y^*by\right)
\\& \quad \leq (1-\nu)^{-1}\left(\nu a - x^*\left({\left\|a^{-1}\right\|}^{-1}e\right)x\right)
+ \nu^{-1}\left((1-\nu)b - y^*\left({\left\|b^{-1}\right\|}^{-1}e\right)y\right)
\\& \quad =(1-\nu)^{-1}\left(\nu\left({\left\|a^{-1}\right\|}^{-1}e\right) - x^*\left({\left\|a^{-1}\right\|}^{-1}e\right)x\right)
+ \nu^{-1}\left((1-\nu)\left({\left\|b^{-1}\right\|}^{-1}\right)e - y^*\left({\left\|b^{-1}\right\|}^{-1}e\right)y\right)
\\& \qquad \quad +{A}_{\nu}\left(\nu^{-1}b, (1-\nu)^{-1}a\right)
- {A}_{\nu}\left(\nu^{-1}{\left\|b^{-1}\right\|}^{-1}e, (1-\nu)^{-1}{\left\|a^{-1}\right\|}^{-1}e\right)
\\& \quad \leq \displaystyle{\max_{x+y=e}}\left\{(1-\nu)^{-1}\left(\nu\left({\left\|a^{-1}\right\|}^{-1}e\right) - x^*\left({\left\|a^{-1}\right\|}^{-1}e\right)x\right)
+ \nu^{-1}\left((1-\nu)\left({\left\|b^{-1}\right\|}^{-1}\right)e - y^*\left({\left\|b^{-1}\right\|}^{-1}e\right)y\right)\right\}
\\& \qquad \quad +{A}_{\nu}\left(\nu^{-1}b, (1-\nu)^{-1}a\right)
- {A}_{\nu}\left(\nu^{-1}{\left\|b^{-1}\right\|}^{-1}e, (1-\nu)^{-1}{\left\|a^{-1}\right\|}^{-1}e\right)
\\& \quad = {C}_{\nu}\left({\left\|a^{-1}\right\|}^{-1}e, {\left\|b^{-1}\right\|}^{-1}e\right)
+{A}_{\nu}\left(\nu^{-1}b, (1-\nu)^{-1}a\right)
- {A}_{\nu}\left(\nu^{-1}{\left\|b^{-1}\right\|}^{-1}e, (1-\nu)^{-1}{\left\|a^{-1}\right\|}^{-1}e\right)
\\& \quad = {A}_{\nu}\left(\nu^{-1}{\left\|b^{-1}\right\|}^{-1}e, (1-\nu)^{-1}{\left\|a^{-1}\right\|}^{-1}e\right)
-{H}_{\nu}\left({\left\|a^{-1}\right\|}^{-1}e, {\left\|b^{-1}\right\|}^{-1}e\right)
\\& \qquad \quad +{A}_{\nu}\left(\nu^{-1}b, (1-\nu)^{-1}a\right)
- {A}_{\nu}\left(\nu^{-1}{\left\|b^{-1}\right\|}^{-1}e, (1-\nu)^{-1}{\left\|a^{-1}\right\|}^{-1}e\right),
\end{align*}
\endgroup
and so
\begingroup\makeatletter\def\f@size{8.9}\check@mathfonts
\begin{align}\label{T.0071.I.1}
(1-\nu)^{-1}\left(\nu a - x^*ax\right) + \nu^{-1}\left((1-\nu)b - y^*by\right)
\leq {A}_{\nu}\left(\nu^{-1}b, (1-\nu)^{-1}a\right)
- {H}_{\nu}\left({\left\|a^{-1}\right\|}^{-1}e, {\left\|b^{-1}\right\|}^{-1}e\right).
\end{align}
\endgroup
Since ${H}_{\nu}\left({\left\|a^{-1}\right\|}^{-1}e, {\left\|b^{-1}\right\|}^{-1}e\right)=
{H}_{\nu}\left({\left\|a^{-1}\right\|}^{-1}, {\left\|b^{-1}\right\|}^{-1}\right)e$, from the inequality \eqref{T.0071.I.1} we get
\begingroup\makeatletter\def\f@size{9}\check@mathfonts
\begin{align*}
(1-\nu)^{-1}\left(\nu a - x^*ax\right) + \nu^{-1}\left((1-\nu)b - y^*by\right)
\leq {A}_{\nu}\left(\nu^{-1}b, (1-\nu)^{-1}a\right)
- {H}_{\nu}\left({\left\|a^{-1}\right\|}^{-1}, {\left\|b^{-1}\right\|}^{-1}\right)e.
\end{align*}
\endgroup
Thus
\begingroup\makeatletter\def\f@size{8}\check@mathfonts
\begin{align}\label{T.0071.I.2}
\displaystyle{\max_{x+y=e}}\left\{(1-\nu)^{-1}\left(\nu a - x^*ax\right) + \nu^{-1}\left((1-\nu)b - y^*by\right)\right\}
\leq {A}_{\nu}\left(\nu^{-1}b, (1-\nu)^{-1}a\right)
- {H}_{\nu}\left({\left\|a^{-1}\right\|}^{-1}, {\left\|b^{-1}\right\|}^{-1}\right)e.
\end{align}
\endgroup
The desired inequality now follows from \eqref{T.0071.I.2} and Theorem \ref{T.01}.
\end{proof}
\textbf{Conflict of interest:} The author declares that he has no conflict of interest.

\textbf{Data availability:} Data sharing not applicable to the present paper as no data sets
were generated or analyzed during the current study.

\textbf{Funding:} Not applicable.
\bibliographystyle{amsplain}

\end{document}